\documentclass{amsart}
\usepackage[normalem]{ulem}

\usepackage{amssymb,epsfig,mathrsfs,mathpazo}
\usepackage{enumerate}
\usepackage{color}
\usepackage{epsfig,bm}
\usepackage[multiple]{footmisc}
\usepackage{accents} 
\usepackage{hyperref}

\usepackage{tikz}
\usepackage{pgfplots}
\usetikzlibrary{calc,math}
\newcommand\convtri[4]{
	\draw[line width = 0.5mm] (axis cs:#1,#2) -- (axis cs:10*#1,#2) -- (axis cs:10*#1,#2 * 0.1^#3) -- cycle;
	\node[right] at (axis cs:10*#1,#2 * 0.1^0.5^#3) {$#4$};}
\newcommand\convtriinv[4]{
	\draw[line width = 0.5mm] (axis cs:#1,#2) -- (axis cs:#1,#2 * 0.1^#3) -- (axis cs:10*#1,#2 * 0.1^#3) -- cycle;
	\node[left] at (axis cs:#1,#2 * 0.1^0.5^#3) {$#4$};}

\newtheorem{theorem}{Theorem}[section]

\newtheorem{algorithm}[theorem]{Algorithm}
\newtheorem{proposition}[theorem]{Proposition}

\newtheorem{corollary}[theorem]{Corollary}

\theoremstyle{definition}

\theoremstyle{remark}
\newtheorem{remark}[theorem]{Remark}

\numberwithin{equation}{section}

\newcommand{\R}{\mathbb R}

\newcommand{\N}{\mathbb N}

\newcommand{\identity}{\mathrm{Id}}

\newcounter{mylistcounter}
\renewcommand{\themylistcounter}{{\bf (\Roman{mylistcounter})}}

\DeclareMathOperator{\ran}{ran}

\DeclareMathOperator{\diam}{diam}
\DeclareMathOperator*{\argmin}{argmin}   

\DeclareMathOperator{\meas}{meas}

\usepackage[T3,T1]{fontenc}
\DeclareSymbolFont{tipa}{T3}{cmr}{m}{n}
\DeclareMathAccent{\invbreve}{\mathalpha}{tipa}{16}

\DeclareMathOperator{\divv}{div}

\newcommand{\be}{\begin{equation}}
\newcommand{\ee}{\end{equation}}

\newcommand{\cP}{{\mathcal P}}
\newcommand{\cS}{{\mathcal S}}
\newcommand{\cI}{{\mathcal I}}
\newcommand{\cJ}{{\mathcal J}}
\newcommand{\cK}{{\mathcal K}}
\newcommand{\RT}{\mathit{RT}}

\newcommand{\TT}{\mathcal T}
\newcommand{\PP}{\mathcal P}

\newcommand*\patchAmsMathEnvironmentForLineno[1]{%
  \expandafter\let\csname old#1\expandafter\endcsname\csname #1\endcsname
  \expandafter\let\csname oldend#1\expandafter\endcsname\csname end#1\endcsname
  \renewenvironment{#1}%
     {\linenomath\csname old#1\endcsname}%
     {\csname oldend#1\endcsname\endlinenomath}}%
\newcommand*\patchBothAmsMathEnvironmentsForLineno[1]{%
  \patchAmsMathEnvironmentForLineno{#1}%
  \patchAmsMathEnvironmentForLineno{#1*}}%
\AtBeginDocument{%
\patchBothAmsMathEnvironmentsForLineno{equation}%
\patchBothAmsMathEnvironmentsForLineno{align}%
\patchBothAmsMathEnvironmentsForLineno{flalign}%
\patchBothAmsMathEnvironmentsForLineno{alignat}%
\patchBothAmsMathEnvironmentsForLineno{gather}%
\patchBothAmsMathEnvironmentsForLineno{multline}%
}
\usepackage[mathlines]{lineno}

\title{Improved rates for a space-time FOSLS of parabolic PDEs}
\date{\today}

\author{Gregor Gantner} 
\address{Institute of Analysis and Scientific Computing, TU Wien, Wiedner Hauptstra{\ss}e 8-10, 1040 Vienna, Austria.}
\email{gregor.gantner@asc.tuwien.ac.at}

\author{Rob Stevenson}
\address{Korteweg-de Vries (KdV) Institute for Mathematics, University of Amsterdam, P.O. Box 94248, 1090 GE Amsterdam, The Netherlands.}
\email{r.p.stevenson@uva.nl}

\thanks{The first author has been supported by the Austrian Science Fund (FWF) under grant J4379-N. The second author has been supported by NSF Grant DMS 172029.}



\begin{document}
\maketitle

\begin{abstract} 
We consider the first-order system space-time formulation of the heat equation introduced in \cite{23.5}, and analyzed in \cite{75.257,75.28}, with solution components
$(u_1,{\bf u}_2)=(u,-\nabla_{\bf x} u)$. The corresponding operator is boundedly invertible between a Hilbert space $U$ and a Cartesian product of $L_2$-type spaces, which facilitates easy first-order system least-squares (FOSLS) discretizations.
Besides $L_2$-norms of $\nabla_{\bf x} u_1$ and ${\bf u}_2$, the (graph) norm of $U$ contains the $L_2$-norm of $\partial_t u_1 +\divv_{\bf x} {\bf u}_2$.
When applying standard finite elements w.r.t.~simplicial partitions of the space-time cylinder, estimates of the approximation error w.r.t.~the latter norm require higher-order smoothness of ${\bf u}_2$. In experiments for both uniform and adaptively refined partitions, this manifested itself 
in disappointingly low convergence rates for non-smooth solutions $u$.

In this paper, we construct finite element spaces w.r.t.~prismatic partitions. They come with 
a quasi-interpolant that satisfies a near commuting diagram in the sense that, apart from some harmless term, the aforementioned error depends exclusively on the smoothness of $\partial_t u_1 +\divv_{\bf x} {\bf u}_2$, i.e., of the forcing term $f=(\partial_t-\Delta_x)u$.
Numerical results show significantly improved convergence rates.
\end{abstract}

\section{Introduction}
\subsection{Space-time variational formulation of the heat equation} \label{sec:2ndorder} This paper is about the numerical solution of second-order parabolic evolution equations in a simultaneous space-time variational formulation.
For convenience only we restrict the discussion to the model problem of the heat equation with homogeneous Dirichlet (lateral) boundary conditions on the time-space cylinder $Q:=I \times \Omega$, where $I:=(0,T)$, and $\Omega \subset \R^d$ is a bounded Lipschitz domain.
The common space-time variational formulation of finding $u$ that, for given data $f,u_0$, satisfies
\begin{align}\label{eq:heat}
\left\{
\begin{array}{rcl} \displaystyle \iint_Q \partial_t u v +\nabla_{\bf x} u \cdot \nabla_{\bf x} v\,dt\,d{\bf x}&=& \displaystyle \iint_Q  f v \,dt\,d{\bf x}\\
\displaystyle \int_\Omega u(0,\cdot) w \,d{\bf x} &=& \displaystyle \displaystyle \int_\Omega u_0 w \,d{\bf x}
\end{array}\right.
\end{align}
for all `test functions' $v,w$, induces a boundedly invertible operator from $X$ to $Y' \times L_2(\Omega)$, where 
$$
X:=L_2(I;H^1_0(\Omega)) \cap H^1(I;H^{-1}(\Omega)), \quad Y:=L_2(I;H^1_0(\Omega))
$$
(see, e.g., \cite{63} or \cite{314.9}). Already because $X \neq Y \times L_2(\Omega)$, the corresponding bilinear form is non-coercive, so that one cannot resort to simple conforming Galerkin discretizations.
See, however, \cite{249.2,169.05} for methods that apply in the case of a homogeneous initial condition.

\subsection{Minimal residual discretization} The approach introduced by R. Andreev in \cite{11} is to consider minimal residual (or least squares) discretizations. 
They amount to minimizing the residual over a selected finite-dimensional trial space $X^\delta \subset X$.
Since the residual of the PDE lives in the dual space $Y'$, however, its norm cannot be evaluated exactly and so has to be approximated by the dual norm over a suitable 
finite dimensional test subspace $Y^\delta \subset Y$.
The resulting minimal residual solution from $X^\delta$ is a quasi-best approximation from $X^\delta$ when the pair $(X^\delta,Y^\delta)$ satisfies
a Ladyshenskaja--Babu\u{s}ka--Brezzi (LBB) condition. 
In view of the application with finite element subspaces w.r.t.~partitions of $Q$ that preferably are as general as possible, it is a drawback that so far the verification of this LBB condition has been restricted to pairs of finite element spaces w.r.t. partitions of $Q$ that permit a decomposition into time-slabs (\cite{249.99,249.992}).

\begin{remark} In \cite{249.991} uniform LBB-stability was demonstrated for trial and test spaces that are spanned by (adaptively generated) sets of tensor products of temporal wavelets and spatial finite element spaces.
\end{remark}

\subsection{FOSLS}\label{sec:FOSLS}
By writing the forcing function $f \in Y'$ (non-uniquely) as $f_1 + \divv_{{\bf x}} {\bf f}_2$ for some $f_1\in Y'$ and ${\bf f}_2\in L_2(Q)^d$, the heat equation for $u=u_1$ can equivalently be written as the \emph{first-order system}
\begin{align}\label{eq:first-order system}
 G{\bf u}:=  \left(\begin{array}{@{}c@{}} \partial_t u_1+\divv_{\bf x} {\bf u}_2  \\ -{\bf u}_2 - \nabla_{\bf x} u_1 \\  u_1(0,\cdot)\end{array} \right) 
  =  \left(\begin{array}{@{}c@{}}  f_1  \\ {\bf f}_2\\ u_0 \end{array} \right),\quad u_1|_{I \times \partial\Omega} = 0.
\end{align}
From the well-posedness of the space-time variational formulation of the second-order equation, and the boundedness of $\divv_{\bf x}\colon L_2(Q)^d \rightarrow Y'$ and $\nabla_{\bf x}\colon X \rightarrow L_2(Q)^d$, one can infer that $G$ is a boundedly invertible operator from
$$
X \times L_2(Q)^d\quad \text{to} \quad Y'  \times L_2(Q)^d \times L_2(\Omega)
$$
(see \cite[Lem.~2.3 \& Rem.~2.4]{243.85}).
Because of the presence of the dual space $Y'$, however, least-squares finite element discretizations of this first-order system suffer from the same restrictions imposed by an LBB condition.

As was first suggested in \cite[\S9.1.4]{23.5}, a solution for this problem is to restrict to functions ${\bf u}=(u_1,{\bf u}_2)$ for which $\divv {\bf u}:=\partial_t u_1+\divv_{\bf x} {\bf u}_2 \in L_2(Q)$.
Doing so, in \cite{75.257} it was proven that $G$ is a boundedly invertible operator from
$$
U := \{{\bf u}=(u_1,{\bf u}_2) \in X \times L_2(Q)^d\colon \divv {\bf u} \in L_2(Q)\},
$$
equipped with the  graph norm, to its range in
$$
L:=L_2(Q)\times L_2(Q)^d \times L_2(\Omega),
$$
i.e., $\|G {\bf u}\|_L \eqsim \|{\bf u}\|_U$ (${\bf u} \in L$)\footnote{In this work, by $C \lesssim D$ we will mean that $C\ge0$ can be bounded by a multiple of $D\ge0$, independently of parameters on which $C$ and $ D$ may depend. 
Obviously, $C \gtrsim D$ is defined as $C \lesssim D$, and $C\eqsim D$ as $C\lesssim D$ and $C \gtrsim D$.}. Consequently, for any ${\bf f}=(f_1,{\bf f}_2,u_0) \in \ran G$, and \emph{any} closed subspace $U^\delta \subset U$, finding ${\bf u}^\delta:=\argmin_{{\bf v} \in U^\delta}\|{\bf f}-G {\bf v}\|_L$ is a well-posed first order system least-squares (FOSLS) discretisation of $G {\bf u}={\bf f}$. Because the $L$-norm can be efficiently exactly evaluated, in \cite{23.5} such a least squares discretisation is referred to as being `practical'.
The resulting ${\bf u}^\delta$ is a quasi-best approximation from $U^\delta$ to the solution ${\bf u}$ in the norm on $U$,
i.e., $\|{\bf u}-{\bf u}^\delta\|_U \lesssim \inf_{{\bf v} \in U^\delta} \|{\bf u}-{\bf v}\|_U$. 

Adding to the results from \cite{75.257}, in \cite{75.28} the aforementioned range of $G$ was shown to be the full space $L$. Furthermore, it was shown that
$$
U := \{{\bf u}=(u_1,{\bf u}_2) \in L_2(I;H^1_0(\Omega)) \times L_2(Q)^d\colon \divv {\bf u} \in L_2(Q)\},
$$
equipped with its  graph norm, gives an equivalent definition of $U$. Finally, assuming that $f \in Y'$ is given as  $f=f_1 + \divv_{{\bf x}} {\bf f}_2$ for some $f_1\in L_2(Q)$ and ${\bf f}_2\in L_2(Q)^d$, which always exist by Riesz' representation theorem, it was shown that the first-order system formulation with spaces $U$ and $L$ applies whenever
the space-time variational formulation of second order from Section~\ref{sec:2ndorder} does. 
We also mention our recent generalization to the instationary Stokes problem with so-called slip boundary conditions~\cite{gs22b}.

A main advantage of the FOSLS approach is that it corresponds to a bilinear form that is symmetric and positive definite which is beneficial in several applications {(cf. \cite{fk22,gs22a}).
Being of least-squares type, the method comes with a built-in efficient and reliable a posteriori estimator for the error in any approximation, being the $L$-norm of its residual. 
The square of this norm naturally splits into contributions associated to the individual elements which can be used to drive an adaptive solution routine. Plain convergence of this routine has been demonstrated in \cite{75.28}, where in practice one observes that this adaptive routine selects a sequence of partitions that results in the best possible convergence rate.

\subsection{Low rates for non-smooth solutions}
For smooth solutions, this FOSLS works fully satisfactorily.
When applying (vectorial) continuous piecewise polynomial approximation of degree $k$ w.r.t.~a quasi-uniform partition of $Q$ into $(d+1)$-\emph{simplices} the error in $U$-norm is of the optimal order $ \text{DoF}^{-\frac{k}{d+1}}$.

Obviously for solutions that are insufficiently smooth, reduced rates are obtained. 
The experiments reported in \cite{75.257}, however, show that for several problems adaptivity hardly improves these rates, which remain disappointingly low.
For example, for the heat equation on $\Omega=(0,1)$,
with forcing function $f(=f_1)=2$ and initial condition $u_0=1$, with a quasi-uniform partition of $Q=I \times \Omega$ into triangles and $k=1$, in \cite{75.257} a rate equal to $0.08$ was reported. 
A low rate was to be expected due to the singularities in the solution at the corners $(0,0)$ and $(0,1)$ of $Q$
  as a consequence of the discontinuity between initial and boundary conditions. But even with adaptivity this rate improved to only $0.17$.
  The observed rate for $\Omega=(0,1)^2$,  $f=0$, and $u_0=1$ and a quasi-uniform partition of $Q$ into tetrahedra was $0.07$, which even did not improve at all by adaptive refinements.
  
  The disappointingly low rates for non-smooth solutions with the FOSLS method whilst applying adaptive refinements can be seen as a prize to be paid for the inclusion of the condition $\divv {\bf u} \in L_2(Q)$ in the definition of the space $U$. Indeed, 
with for example linear trial spaces for $u_1$ and ${\bf u}_2$, the (local) approximation error in ${\bf u}$ measured in $\|\divv\cdot\|_{L_2(Q)}$, being part of the graph norm on $U$, is of the order of the (local) mesh-size times the maximum of the (local) $H^2$ semi-norms of $u_1$ and ${\bf u}_2$. Realizing that $u_1=u$, and, for ${\bf f}_2=0$,  ${\bf u}_2=-\nabla_{\bf x} u_1$, we conclude that this (local) approximation error is of the order of the (local) mesh-size times  the maximum of the (local) $L_2$-norm of second- and \emph{third}- order derivatives of $u$. 
Similar bounds on the (local) approximation errors in $u_1$ and ${\bf u}_2$ in $L_2(I;H^1_0(\Omega))$- or $L_2(Q)^d$-norms involve at most \emph{second} order derivatives of $u$. Obviously, similar considerations apply to finite element spaces of higher order.

\subsection{Towards a remedy}
To cure this problem, an idea is to apply finite element spaces that are specially designed for the space $H(\divv;Q)$, as e.g.~the Raviart--Thomas finite elements.
Thanks to their `commuting diagram property', for, say, $\RT_0$-elements the (local) approximation error in ${\bf u}$ measured in $\|\divv\cdot\|_{L_2(Q)}$
 is of the order of the (local) mesh-size times the (local) $H^1$ semi-norm of $\divv {\bf u}$, the latter being equal to $f_1$. 
In many cases, the condition that $f_1$ has some smoothness is much milder than the corresponding condition on the smoothness of both terms $\partial_t u_1$ and $\divv_{\bf x} {\bf u}_2$ separately.
Unfortunately, because of the condition $u_1 \in L_2(I;H^1_0(\Omega))$ incorporated in the definition of $U$, $H(\divv;Q)$-conformity of the finite element space does not suffice, and therefore the Raviart--Thomas elements are not applicable.

In \cite{38.77} the space of vectorial continuous piecewise linears was enriched with bubbles associated to the faces of the simplices
such that the range of the divergence operator applied to this trial space equals the space of piecewise constants. 
This trial space, which is $H^1(Q)$- and thus $U$-conforming, 
comes with a quasi-interpolator (cf.~also \cite{35.927}) that satisfies a commuting diagram so that the (local) approximation error in ${\bf u}$ measured in $\|\divv\cdot\|_{L_2(Q)}$
is of the order of the (local) mesh-size times the local $H^1$ semi-norm of $\divv {\bf u}$.
Unfortunately, in our setting, where ${\bf u}=(u_1,{\bf u}_2)$ and ${\bf u}_2=-({\bf f}_2+\nabla_{\bf x} u_1)$, 
it turns out that with this space the problem of the appearance of higher order derivatives of $u$ in the local error bounds manifests itself at another place.
Due to the fact that not all faces of a $(d+1)$-simplex inside $Q$ are either parallel or perpendicular to the bottom $\{0\}\times\Omega$ of $Q$,
the local quasi-interpolation error in $u_1$ in the $L_2 \otimes H^1$-norm depends on derivatives of ${\bf u}_2$, and thus on derivatives of $u$ that are of one order higher than desirable.
Numerical experiments with the FOSLS method for this trial space do not show better rates than without the inclusion of these bubbles

\subsection{The approach from this work}
We consider finite element subspaces of $U$ w.r.t.~prismatic partitions of the time-space cylinder $Q$.
The space of local shape functions on each prism will be a Cartesian product of tensor products of finite element spaces in time and space.
In Section~\ref{sec:localfem}, we construct a local quasi-interpolant that satisfies the desired commuting diagram.

Defining the global quasi-interpolant piecewise as this local quasi-interpolant would correspond to a non-conforming finite element space, as the $u_1$-component would be discontinuous over element interfaces in the spatial direction and therefore not in $L_2(I;H^1_0(\Omega))$.
By constructing a global quasi-interpolant by averaging the local quasi-interpolants over these interfaces, the resulting global finite element space is conforming.
Although consequently a truly commuting diagram is lost,  in Section~\ref{sec:conforming}, for locally conforming prismatic partitions, we show a local upper bound for the global interpolation error that is satisfactory in the following sense: For the same convergence rate, it requires \emph{less} local smoothness of the solution than a corresponding interpolation error bound with simplicial finite element spaces.

Having the aim to allow for local (adaptive) refinements, appropriate for non-smooth solutions, \emph{nonconforming} prismatic partitions have to be considered.
In Section~\ref{sec:general prisms}, we demonstrate that also on local patches which are nonconforming, 
the local upper bound for the global interpolation error has the right order, albeit under stronger smoothness requirements on the solution.

The numerical experiments for an adaptive solver presented in Section~\ref{sec:numres} demonstrate the advantage of the use of prismatic elements for various singular solutions compared to the use of standard simplicial elements.

We also mention that prismatic partitions provide the possibility to use different (local) mesh-sizes in time and space.
For an ultra-weak DPG reformulation of the system \eqref{eq:first-order system} introduced in \cite{64.585}, a so-called parabolic scaling of the temporal and spatial mesh-sizes turns out to be beneficial for certain singular solutions, demonstrating the potential usefulness of anisotropic prisms.
However, while some theoretical results in this work apply to anisotropic prismatic meshes as well, the main results and the numerical experiments are restricted to isotropic prismatic partitions.

\section{The finite element} \label{sec:localfem}
As \emph{element domain} we consider a \emph{prism} $P=J \times K$, where $J$ is a (closed) interval and $K$ a (closed) $d$-simplex.
We will use the notations $h_J:=\diam J$ and $h_K:=\diam K$.

For $\ell, k \in \N_0$, we consider the \emph{space of shape functions} given by
$$
\cS_{\ell,k}(P):=\big(\cP_{\ell+1}(J) \otimes \cP_k(K)\big) \times \big(\cP_\ell(J) \otimes \RT_k(K)\big),
$$
with $\RT_k(K)$ being the Raviart--Thomas element of order $k$, see, e.g., \cite[III.3]{35.855}.
For $d=1$, this space should be read as $\cP_{k+1}(K)$.

The definition of the finite element is completed by the specification of the \emph{degrees of freedom} (DoFs) for both $\cP_{\ell+1}(J) \otimes \cP_k(K)$ and $\cP_\ell(J) \otimes \RT_k(K)$.
As DoFs for $\cP_{\ell+1}(J)$ we take
\be \label{dof1}
p(t) \quad(t \in \partial J)\quad \text{and, when } \ell \geq 1,\quad \int_J p q \,dt \quad(q \in \cP_{\ell-1}(J)),
\ee
and as DoFs for $\cP_{k}(K)$,
$$
\int_K p q \,d{\bf x}\quad(q \in \cP_k(K)).
$$
The DoFs for $\cP_{\ell+1}(J) \otimes \cP_k(K)$ are the tensor products of the DoFs for both factors.

As DoFs for $\cP_{\ell}(J)$ we take
\be \label{dof2a}
\int_J p q \,dt \quad(q \in \cP_{\ell}(J)),
\ee
and as DoFs for $\RT_k(K)$,
\be \label{dof2}
\int_{\partial K} {\bf p}\cdot {\bf n} \,q\,d{\bf s} \quad(q \in R_k(\partial K)) \quad \text{and, when } k \geq 1,\quad \int_K {\bf p}\cdot {\bf q} \,d{\bf x} \quad({\bf q} \in \cP_{k-1}(K)^d),
\ee
where $R_k(\partial K)=\{q \in L_2(\partial K)\colon q|_e \in P_k(e) \text{ for all facets } e \text { of } K\}$.
Again, the DoFs for the tensor product $\cP_\ell(J) \otimes \RT_k(K)$ are the tensor products of the DoFs for both factors.

\begin{remark}
Notice that the DoFs for $\cP_{\ell+1}(J)$ were chosen so that in a certain sense this space plays the role of a one-dimensional Raviart--Thomas space of order $\ell$.
A difference, however, is that, for $d>1$,  $\RT_k(K) \subsetneq \cP_{k+1}(K)^d$.
\end{remark}

\subsection{The local interpolant and a commuting diagram}
Above DoFs for $\cS_{\ell,k}(P)$ are well-defined on $\big(C(J) \otimes L_1(K)\big) \times \big(L_1(J) \otimes \big(L_s(K)^d \cap H(\divv;K)\big)\big)$ whenever $s>2$, see \cite[\S17.1-2]{70.97}.\footnote{For $d=1$, it suffices to take $s=2$, so that $L_2(K)^d \cap H(\divv;K)=H^1(K)$.}
On such a space they define a \emph{quasi-interpolator} $\cI_{\ell,k}^P=\cI_{\ell,k}^{P,1}\times \cI_{\ell,k}^{P,2}\colon {\bf v} \mapsto \cI_{\ell,k}^{P} {\bf v} \in \cS_{\ell,k}(P)$ by imposing that $\cI_{\ell,k}^{P} {\bf v}$ has the same DoFs as ${\bf v}$. By construction, this $\cI_{\ell,k}^P$ is thus a projector onto $\cS_{\ell,k}(P)$.

\begin{proposition}[Commuting diagram property] \label{prop:commuting} With $Q^P_{\ell,k}$ denoting the $L_2(P)$-orthogonal projector onto $\cP_\ell(J) \otimes \cP_k(K)$, it holds that 
$$
\divv \cI_{\ell,k}^P=Q^P_{\ell,k} \divv.
$$
\end{proposition}

\begin{proof} Clearly $\ran \divv \cI_{\ell,k}^P \subset \cP_\ell(J) \otimes \cP_k(K)$.
For $p \in \cP_\ell(J)$, $q \in  \cP_k(K)$ , and ${\bf v}=(v_1,{\bf v}_2)$, and with ${\bf n}=(n_t,{\bf n}_{\bf x})$ the outward pointing normal vector on $\partial P$, integration by parts gives
\begin{align*}
&\int_J \int_K \divv ({\bf v}- \cI_{\ell,k}^P{\bf v})(t,{\bf x}) p(t) q({\bf x})\,dt\,d{\bf x}=\\
&\quad-\int_J \int_K (v_1-\cI^{P,1}_{\ell,k} v_1) p'(t) q({\bf x})+ ({\bf v}_2-\cI^{P,2}_{\ell,k} {\bf v}_2)\cdot p(t)\nabla q({\bf x}) \,dt\,d{\bf x}+\\
&\quad
\int_K \int_{\partial J} (v_1-\cI^{P,1}_{\ell,k} v_1) n_t  p(s) q({\bf x})\,ds \,d{\bf x}+
\int_J \int_{\partial K} ({\bf v}_2-\cI^{P,2}_{\ell,k} {\bf v}_2) \cdot {\bf n}_{{\bf x}}  p(t) q({\bf y})\,d{\bf y} \,dt=0
\end{align*}
by the definition of the DoFs. 
This completes the proof.
\end{proof}

\begin{corollary} \label{corol1} For $k' \in \{0,\ldots,k\}$, $k'' \in \{0,\ldots,k+1\}$, $\ell' \in \{0,\ldots,\ell+1\}$, it holds that 

\noindent {\rm (i)} $\displaystyle \|\divv ({\bf v} -\cI_{\ell,k}^P{\bf v})\|_{L_2(P)}  \lesssim h_J^{\ell'} |\divv {\bf v}|_{H^{\ell'}(J)\otimes L_2(K)}
+ h_K^{k''} |\divv {\bf v}|_{L_2(J)\otimes H^{k''}(K)}$,

\noindent {\rm (ii)} $\displaystyle\|\nabla_{\bf x}(v_1 -\cI^{P,1}_{\ell,k}v_1)\|_{L_2(P)^d} \lesssim h_J^{\ell'+1} |v_1|_{H^{\ell'+1}(J)\otimes H^1(K)}
+ h_K^{{k'}} |v_1|_{L_2(J)\otimes H^{{k'}+1}(K)}$,

\noindent  {\rm (iii)} $\displaystyle \|{\bf v}_2 -\cI^{P,2}_{\ell,k}{\bf v}_2\|_{L_2(P)^d} \lesssim h_J^{\ell'} |{\bf v}_2|_{H^{\ell'}(J)\otimes L_2(K)^d}
+ h_K^{{k'}+1} |{\bf v}_2|_{L_2(K)\otimes H^{{k'}+1}(K)^d}$,

\noindent with hidden constants depending only on the shape regularity of $K$, the space dimension $d$, and the polynomial degrees $\ell$ and $k$.
\end{corollary}



\begin{proof} Let $\rho_\ell^{J}$, $\rho_{k}^{K}$ denote the projectors on $\cP_{\ell+1}(J)$ and $\RT_k(K)$ defined by the DoFs from \eqref{dof1} and \eqref{dof2}, respectively, and
 let
$Q_{\ell}^{J}$ and $Q_{k}^{K}$ denote the $L_2$-orthogonal projectors onto $\cP_{\ell}(J)$ and $\cP_{k}(K)$, respectively, so that
$$
\cI^{P,1}_{\ell,k}=\rho_\ell^{J}\otimes Q_k^{K},\quad \cI^{P,2}_{\ell,k}=Q_{\ell}^{J}\otimes \rho_{k}^{K}.
$$
It is known, see, e.g.,~\cite[Thm.~16.4]{70.97}, that for $s \in [1,\infty]$,
\be \label{1}
\|(\identity-\rho_k^K) {\bf w}\|_{L_s(K)^d} \lesssim h_K^{k'+1} |{\bf w}|_{W^{k'+1}_s(K)^d} \quad ({\bf w} \in W^{k'+1}_s(K)^d). \footnotemark
\ee
\footnotetext{The case $s \neq 2$ will be relevant later.}
Similarly, we have $\|(\identity-\rho_\ell^{J}) w\|_{L_2(J)}  \lesssim h_J^{\ell'+1} |w|_{H^{\ell'+1}(J)}$ $(w \in H^{\ell'+1}(J))$.

Thanks to Proposition~\ref{prop:commuting}, by writing
\begin{align*}
\identity\otimes \identity-Q_{\ell}^{J}\otimes Q_k^{K}&=(\identity-Q_{\ell}^{J})\otimes \identity +Q_k^{K}\otimes (\identity-Q_k^{K}),\\
\identity\otimes \identity-\rho_\ell^{J}\otimes Q_k^{K}&=(\identity-\rho_\ell^{J})\otimes Q_k^{K}+\identity \otimes (\identity-Q_k^{K}),\\
\identity\otimes \identity-Q_{\ell}^{J}\otimes \rho_{k}^{K}&=(\identity-Q_{\ell}^{J})\otimes \identity+Q_{\ell}^{J} \otimes (\identity-\rho_{k}^{K}),
\end{align*}
standard estimates for the orthogonal projectors give the results.
\end{proof}

\begin{remark} For simplicity Corollary~\ref{corol1} is restricted to $L_2$-based Sobolev norms with integer smoothness indices. Error bounds in terms of general $L_p$-based Sobolev-norms, possibly with fractional smoothness indices can be derived as well.
\end{remark}

In the \emph{isotropic} case of $h_J \eqsim h_K$, the best choice for the polynomial degrees $\ell$ and $k$ is $\ell+1=k$. Indeed, increasing either $\ell$ or $k$ will not improve 
the estimate of minimal order among the upper bounds of Corollary~\ref{corol1}. By taking $\ell'=k=k''$, $\ell'+1=k=k'$, and $\ell'=k=k'+1$ in the bounds (i), (ii), and (iii), respectively, which minimizes the regularity conditions without affecting the minimal order, one arrives at the following upper bound for the interpolation error in the local $U$-norm:

\begin{corollary} With 
\be\label{eq:local Unorm}
\|{\bf v}\|_{U,P}:=\sqrt{\|\nabla_{\bf x}v_1\|^2_{L_2(P)^d}+\|{\bf v}_2\|_{L_2(P)^d}^2+\|\divv {\bf v}\|_{L_2(P)}^2},
\ee 
for $h_P:=h_J \eqsim h_K$ and $k \in \N$, it holds that
$$
\|{\bf v} -\cI_{k-1,k}^P {\bf v}\|_{U,P} \lesssim h_P^k \big(|\nabla_{\bf x} v_1|_{H^k(P)^d}+|{\bf v}_2|_{H^k(P)^d}+
|\divv {\bf v}|_{H^k(P)}\big).
$$

In particular, for ${\bf u}=(u_1,{\bf u}_2)$ being the solution of \eqref{eq:first-order system}, i.e., $u_1=u$ being the solution of the heat equation, it holds that
\be \label{9}
\|{\bf u} -\cI_{k-1,k}^P {\bf u}\|_{U,P} \lesssim h_P^k \big(|\nabla_{\bf x} u|_{H^k(P)^d}+|{\bf f}_2|_{H^k(P)^d}+
|f_1|_{H^k(P)}\big).
\ee
\end{corollary}

\begin{remark}
For comparison, 
to obtain a local error bound of order $k$ with a trial space on a $(d+1)$-simplex $T$, one needs a trial space that contains $\cP_k(T)^{d+1}$.
Then, with $h_T:=|T|^{\frac{1}{d+1}}$, the bound on the local interpolation error in the $\|\cdot\|_{U,P}$-norm 
reads as $h_T^k |{\bf v}|_{H^{k+1}(T)^{d+1}}$, and so for ${\bf v}={\bf u}$, 
as 
$$
h_T^k (|u|_{H^{k+1}(T)}+|\nabla_{\bf x}u|_{H^{k+1}(T)^d}+|{\bf f}_2|_{H^{k+1}(T)^d})
$$ 
which requires more smoothness of $u$ than the bound \eqref{9}.
As was already mentioned in the introduction, for these simplicial elements the addition of bubbles in order to realize a commuting diagram for the divergence operator does not give rise to better results.
\end{remark}

\section{Global finite element space}\label{sec:global space}
\subsection{Conforming partition into prisms}\label{sec:conforming}
Let $\cP$ be an (essentially) disjoint partition of the time-space cylinder $\overline{Q}$ into prisms $P=J \times K$ for closed intervals $J \subset I$ and uniformly shape-regular closed $d$-simplices $K \subset \Omega$.
In this subsection we assume that $\cP$ is \emph{conforming} in the sense that if the intersection of two different prisms from $\cP$ has a positive $d$-dimensional measure, then these prisms share a common facet.
This means that $\cP$ has to be a Cartesian product of an essentially disjoint partition $\cJ$ of $\overline{I}$ into subintervals $J$ and an essentially disjoint conforming partition $\cK$ of $\overline{\Omega}$ into closed $d$-simplices $K$.
For $K \in \cK$, we set \mbox{$\omega_K=\omega_K^\cK:=\cup\{K'\!\! \in \!\cK\colon K' \cap K \neq \emptyset\}$.}

Given $\ell \in \N_0$, $k \in \N$
, we define the global finite element space
$$
\cS_{\ell,k}(\cP)=\{{\bf v}\in U\colon {\bf v}|_P \in \cS_{\ell,k}(P) \,(P \in \cP)\}.
$$
Our aim is to construct a suitable interpolation operator into $\cS_{\ell,k}(\cP)$. For some $s>2$ ($s \geq 2$ when $d=1$), let ${\bf v} \in C(\overline{I};L_1(\Omega)) \times L_1\big(I;L_s(\Omega)^d \cap H(\divv;\Omega)\big)$.
The piecewise interpolation operator 
$$
\cI_{\ell,k}^{\cP}=\cI_{\ell,k}^{\cP,1}\times \cI_{\ell,k}^{\cP,2}:= {\bf v}\mapsto (\cI_{\ell,k}^{P,1} {v_1}|_{P},\cI_{\ell,k}^{P,2} {{\bf v}_2}|_{P})_{P \in \cP}
$$
is, although it maps into $H(\divv;Q)$, not appropriate since $\ran \cI_{\ell,k}^{\cP,1} \not\subset L_2(I;H^1_0(\Omega))$.
Recalling that $\cI_{\ell,k}^{P,1}=\rho_\ell^J \otimes Q_k^K$, 
 we therefore replace $\cI_{\ell,k}^{\cP}$ by
\be \label{7}
 \breve{\cI}_{\ell,k}^{\cP}= \breve{\cI}_{\ell,k}^{\cP,1} \times  \cI_{\ell,k}^{\cP,2}
\ee
with $\breve{\cI}_{\ell,k}^{\cP,1}$ being defined by the replacement of the piecewise orthogonal projector $w \mapsto (Q_k^K w|_K)_K$ by the operator $\breve{Q}_k$ defined as this piecewise orthogonal projector followed by 
averaging at the nodal points\footnote{The set of nodal points is defined as the union over $K$ in the partition of $\Omega$ of the principal lattice of order $k$ on $K$.} shared by multiple $d$-simplices $K$ in the partition of $\Omega$, or setting the value to zero at nodal points on $\partial\Omega$.
Since a piecewise polynomial on $\Omega$ is in $H^1_0(\Omega)$ if (and only if) it is continuous and vanishes at $\partial\Omega$, the mapping $\breve{\cI}_{\ell,k}^{\cP}$ is a projector onto $\cS_{\ell,k}(\cP)$.

It holds that for $K \in \cK$, $\|\breve{Q}_k w\|_{L_2(K)} \lesssim \|w\|_{L_2(\omega_K)}$, and for $n \in \{0,1\}$, $m \in \{1,\ldots,k+1\}$, and $w$ that vanishes at $\partial\Omega$,
\be \label{20}
|({\rm Id}-\breve{Q}_k)w|_{H^n(K)} \lesssim h_K^{m-n} |w|_{H^{m}(\omega_K)}.
\ee
The latter inequality follows by the usual homogeneity and polynomial reproduction arguments after realizing that, when applied to $w$ with $w|_{\partial\Omega}=0$, the definition of $\breve{Q}_k$ does not change when the zero DOFs associated to nodal points on $\partial\Omega$ are replaced by the usual Scott--Zhang functionals in terms of a weighted integrals of $w|_{\partial\Omega}$ (see \cite[Appendix]{75.527}). 
A comprehensive analysis of quasi-interpolators constructed by averaging can be found in \cite{70.965}.

 \begin{proposition} \label{prop1} For any $k' \in \{0,\ldots,k\}$, $k'' \in \{0,\ldots,k+1\}$, $\ell' \in \{0,\ldots,\ell+1\}$, and $P \in \cP$ it holds that 
\begin{align*}
\|\divv ({\bf v} -\breve{\cI}_{\ell,k}^{\cP} {\bf v})\|_{L_2(P)}  &\lesssim  h_J^{\ell'} |\divv {\bf v}|_{{H^{\ell'}(J)\otimes L_2(K})}\\
&\quad + h_K^{k''}  \big( |\divv {\bf v}|_{L_2(J)\otimes H^{k''}(K)}+  | v_1|_{H^1(J)\otimes H^{k''}(\omega_K)}\big),\\
\|\nabla_{\bf x} (v_1 -\breve{\cI}_{\ell,k}^{\cP,1} v_1)\|_{L_2(P)^d} & \lesssim  h_J^{\ell'+1} |v_1|_{H^{\ell'+1}(J)\otimes H^1(\omega_K)}
+ h_K^{{k'}} | v_1|_{L_2(J)\otimes H^{{k'}+1}(\omega_K)},\\
\|{\bf v}_2 -\cI_{\ell,k}^{\cP,2} {\bf v}_2\|_{L_2(P)^d} & \lesssim  h_J^{\ell'} |{\bf v}_2|_{H^{\ell'}(J)\otimes L_2(K)^d}
+ h_K^{{k'}+1} |{\bf v}_2|_{L_2(J)\otimes H^{{k'}+1}(K)^d},
\end{align*}
for the first and second estimate assuming that $v_1$ vanishes at $I \times \partial\Omega$.
The hidden constants depend only on the shape regularity of $\cK$, the space dimension $d$, and the polynomial degrees $\ell$ and $k$.
\end{proposition}

\begin{proof} The third estimate follows directly from Corollary~\ref{corol1}{\rm (iii)}.

Writing $\identity\otimes \identity-\rho_\ell^{J}\otimes \breve{Q}_k=(\identity-\rho_\ell^{J})\otimes \breve{Q}_k+\identity \otimes (\identity-\breve{Q}_k)$,
the second estimate follows from the estimate for $\identity-\rho_\ell^{J}$ that was already used in the proof of Corollary~\ref{corol1}, and, for $w$ that vanishes on $\partial\Omega$, $|\breve{Q}_k w|_{H^1(K)} \lesssim |w|_{H^1(\omega_K)}$, and $|(\identity-\breve{Q}_k)w|_{H^1(K)} \lesssim h_K^{{k'}} |w|_{H^{{k'}+1}(\omega_K)}$.

Writing $(\divv \breve{\cI}_{\ell,k}^{\cP}{\bf v})|_P=\divv \cI_{\ell,k}^{P}{\bf v}|_P+(\partial_t \rho_\ell^J \otimes (\breve{Q}_k-Q_k^K)v_1)|_P$, 
the first estimate follows from Corollary~\ref{corol1}(i), $\|(\rho_\ell^J w)'\|_{L_2(J)}=\|Q_\ell^J w'\|_{L_2(J)} \lesssim \|w'\|_{L_2(J)}$, and, for $w$ that vanishes on $\partial\Omega$, $\|(\breve{Q}_k-Q_\ell^K)w\|_{L_2(K)} \lesssim h_K^{k''} |w|_{H^{k''}(\omega_K)}$.
\end{proof}

By taking $\ell+1=k$, and $\ell'=k=k''$, $\ell'+1=k=k'$, and $\ell'=k=k'+1$ in the first, second, and third bound of Proposition~\ref{prop1}, respectively, we arrive at the following upper bound for the interpolation error in the local $U$-norm~\eqref{eq:local Unorm} in the \emph{isotropic case}:

\begin{corollary} \label{corol2} For $h_P:=h_J \eqsim h_K$ and $k \in \N$, it holds that
\be \label{6}
\|{\bf v} -\breve{\cI}_{k-1,k}^{\cP} {\bf v}\|_{U,P} \lesssim h_P^k \big(|\nabla_{\bf x} v_1|_{H^k(J \times \omega_K)^d}+|{\bf v}_2|_{H^k(P)^d}+
|\divv {\bf v}|_{H^k(P)}\big)
\ee
assuming that $v_1$ vanishes at $I \times \partial\Omega$\mbox{ }\footnote{and is of course such that the norms at the right hand side are bounded. Silently such an assumption has been made at similar instances.}.

In particular, for ${\bf u}$ being the solution of \eqref{eq:first-order system}, it holds that
$$
\|{\bf u} -\breve{\cI}_{k-1,k}^{\cP} {\bf u}\|_{U,P} \lesssim h_P^k \big(|\nabla_{\bf x} u|_{H^k(J \times \omega_K)^d}+|{\bf f}_2|_{H^k(P)^d}+
|f_1|_{H^k(P)}\big).
$$
\end{corollary}

\begin{remark}\label{rem:ell and k}
 The upper bound on $\|\divv ({\bf v} -\breve{\cI}_{\ell,k}^{\cP} {\bf v})\|_{L_2(P)}$ from Proposition~\ref{prop1} does not depend on a norm of ${\bf v}_2$, but it does contain a term $h_K^{k''} |v_1|_{H^1(J)\otimes H^{k''}(\omega_K)}$ meaning that the global interpolant $\breve{\cI}_{\ell,k}^{\cP}$ does \emph{not} give rise to a (truly) commuting diagram. Considering the case that $\ell+1=k$, this upper bound for $\ell'=k=k''$ of order $h_J^k+h_K^k$ involves an $L_2$-norm of a partial derivative of $v_1$ of total order $k+1$, of which one order is in time and the other $k$ orders are in space.

Still for  $\ell+1=k$, the upper bound on $\|\nabla_{\bf x} (v_1 -\breve{\cI}_{k-1,k}^{\cP,1} v_1)\|_{L_2(P)^d}$ for $\ell'+1=k=k'$ from Proposition~\ref{prop1} of order $h_J^k+h_K^k$
involves two $L_2$-norms of partial derivatives of $v_1$ of total order $k+1$, one that has $k$ orders in time and one in space, and the other that has all $k+1$ orders in space.
For $v_1=u$ being a (non-smooth) solution of the heat equation with a (relatively) smooth forcing term $f$, temporal derivatives of $u$  can be expected to grow faster towards a singularity than spatial derivatives of the same order do. In view of this, the additional term in the upper bound of the interpolation error as a consequence of the absence of a truly commuting diagram is harmless.
Because of this, in Corollary~\ref{corol2}, we have bounded the three (local) $L_2$-norms of partial derivatives of $v_1$ from the upper bounds of Proposition~\ref{prop1}
by the single term $|\nabla_{\bf x} v_1|_{H^k(J \times \omega_K)^d}$.
\end{remark}

\subsection{General partition into prisms}\label{sec:general prisms}
When having the aim to allow local (adaptive) refinements, it is needed to consider prismatic partitions that are \emph{nonconforming}.
In this section, we therefore consider (essentially) disjoint partitions $\cP$ of the time-space cylinder $\overline{Q}$ into prisms $P=J \times K$, for closed intervals $J \subset I$ and uniformly shape-regular closed $d$-simplices $K \subset \Omega$, such that for any two different $P, P' \in \cP$, where $P'=J' \times K'$, if $\meas_d(P \cap P')>0$, then $|J| \eqsim |J'|$ and $|K| \eqsim |K'|$, 
and either $P$ and $P'$ share a common facet, or $P \cap P'$ is a complete facet $F$ of one the two prisms, say $P$, and a proper subset of a facet $F'$ of $P'$.
In the latter case we refer to $P'$ (or $F'$) as the \emph{master} of the \emph{slave} $P$ (or $F$).
We assume that $\cP$ is \emph{$1$-irregular} in the sense that if $F' \subset P'$ is a master facet, then it has non-empty intersection with possible slave facets of $P'$. 
A suitable refinement strategy generating partitions that satisfy all given requirements will be presented in Section~\ref{sec:adaptive algorithm}.
Facets that are parallel to the bottom $\{0\}\times \Omega$ of the time-space cylinder will be referred to as horizontal facets, and the other ones as lateral facets.
As in the case of having a conforming prismatic partition, we set
$$
\cS_{\ell,k}(\cP):=\{{\bf v}\in U\colon {\bf v}|_P \in \cS_{\ell,k}(P) \,(P \in \cP)\}.
$$

We have to adapt the definition of the quasi-interpolant $\breve{\cI}^{\cP}_{\ell,k}=\breve{\cI}_{\ell,k}^{\cP,1} \times \cI_{\ell,k}^{\cP,2}$ from \eqref{7} to generally nonconforming partitions.
We will construct an adapted quasi-interpolant, denoted by $\invbreve{\cI}^{\cP}_{\ell,k}=\invbreve{\cI}_{\ell,k}^{\cP,1} \times \invbreve{\cI}_{\ell,k}^{\cP,2}$,
that is a projector onto $\cS_{\ell,k}(\cP)$ such that, for $h_J \eqsim h_K$ and $\ell+1=k$ ($\ell=k \geq 1$ when $d=1$), an estimate like \eqref{6}  from Corollary~\ref{corol2}  is valid, albeit with stronger norms on ${\bf v}=(v_1,{\bf v}_2)$ in the upper bound.  

The definition $\invbreve{\cI}^{\cP}_{\ell,k}$ will show that for all $P \in \cP$ for which the local patch $\omega_P=\omega_P^\cP:=\cup\{P'\in \cP\colon P' \cap P \neq \emptyset\}$ is conforming, it holds that $(\invbreve{\cI}^{\cP}_{\ell,k} {\bf v})|_P=(\breve{\cI}^{\cP}_{\ell,k}{\bf v})|_P$ so that the favourable estimate \eqref{6} from Corollary~\ref{corol2} remains valid.


Although for the other $P\in \cP$ we will assume locally sufficient additional regularity of ${\bf v}$, a proof of the interpolation error in local $\|\cdot\|_{U,P}$-norm
being of order $h_P^k$ is nevertheless not immediate. Indeed, because $\cP_k(J \times K)^{d+1}$ is not included in our local finite element space $\cS_{k-1,k}(P)$, it is not a consequence of the Bramble--Hilbert lemma.\footnote{Literature on Raviart--Thomas or similar $H(\divv)$-elements w.r.t. nonconforming partitions seems scarce. In \cite{9.3}, $\RT_k$-elements w.r.t.~nonconforming quadrilateral partitions are considered. Interpolation error estimates of optimal order w.r.t.~$L_2$-norms are given (even with an explicit dependence of constants on $k$).
Such estimates for $H(\divv)$-norms are however missing (with the exception of \cite[Thm.~15]{9.3} based on a commuting diagram which, however, is not valid for nonconforming partitions). }
\bigskip

\subsubsection{Definition and analyis of $\invbreve{\cI}_{\ell,k}^{\cP,1}$}
Given $\ell\in \N_0$, $k \in \N$, we define the set of nodal points in $P \in \cP$ as the Cartesian product of the principal lattice of order $\ell+1$ on $J$ and the principal lattice of order $k$ on $K$.
A nodal point $\nu$ in $P$ that is not on a slave facet of $P$ is called a \emph{free node} in $P$.
Notice that a function in $\cS_{\ell,k}(P)$ is uniquely determined by its values in the nodal points in $P$, and that the restriction of such a function to a facet of $P$ is determined by its  values in the nodal points on that facet.

Let $v_1 \in C(\overline{I};L_1(\Omega))$, and let $\nu$ be a node of $P$. For $\nu \in P\setminus\partial P$, we set $(\invbreve{\cI}_{\ell,k}^{\cP,1} v_1)(\nu)$ $:=(\cI_{\ell,k}^{P,1} v_1)(\nu)$, and
for $\nu \in \overline{I} \times \partial\Omega$, we set $(\invbreve{\cI}_{\ell,k}^{\cP,1} v_1)(\nu):=0$.
For $\nu \in \partial P \setminus \partial\Omega$ being a free node in $P$, we define $(\invbreve{\cI}_{\ell,k}^{\cP,1} v_1)(\nu)$ as the average of $(\cI_{\ell,k}^{P',1} v_1)(\nu)$ over all $P' \in \cP$ in which $\nu$ is a free node.
In the remaining case of $\nu$ being on a slave facet $F$ of $P$ with master facet $F'$ of $P'$, 
we set $(\invbreve{\cI}_{\ell,k}^{\cP,1} v_1)(\nu)$ equal to the value at $\nu$ of $(\invbreve{\cI}_{\ell,k}^{\cP,1} v_1)|_{F'}$.
Note that the condition of $\cP$ being $1$-irregular ensures that all nodes of $P'$ that are on $F'$ are free so that the latter term is well-defined.

By construction, $\invbreve{\cI}_{\ell,k}^{\cP,1} v_1$ is continuous and vanishes at $I \times \partial\Omega$.
The same arguments used to demonstrate \eqref{20} together with standard inverse inequalities show that
for $P=J \times K$, $j \in \{0,\ldots,\min(\ell+1,k)\}$, and $v_1$ that vanishes at $I \times \partial\Omega$,
\begin{align} \label{10a}
\|\partial_t(v_1-\invbreve{\cI}_{\ell,k}^{\cP,1} v_1)\|_{L_2(P)} &\lesssim h_J^{-1} \max(h_J,h_K)^{j+1} |v_1|_{H^{j+1}(\omega_P)},\\ \label{11}
\|\nabla_{\bf x}(v_1-\invbreve{\cI}_{\ell,k}^{\cP,1} v_1)\|_{L_2(P)^d} &\lesssim h_K^{-1}\max(h_J,h_K)^{j+1} |v_1|_{H^{j+1}(\omega_P)},
\end{align}
Because we are going to estimate 
\be \label{12}
\begin{split}
\|\divv ({\bf v}&-\invbreve{\cI}_{\ell,k}^{\cP}{\bf v})\|_{L_2(P)} \leq \|\divv ({\bf v}-\cI_{\ell,k}^{P}{\bf v}|_P)\|_{L_2(P)} \\
&+\|\partial_t(\cI_{\ell,k}^{P,1} v_1|_{P}-\invbreve{\cI}_{\ell,k}^{\cP,1} v_1)\|_{L_2(P)}
+\|\divv_{\bf x}(\cI_{\ell,k}^{P,2} {\bf v}_2|_{P}-\invbreve{\cI}_{\ell,k}^{\cP,2} {\bf v}_2)\|_{L_2(P)},
\end{split}
\ee
 here we notice that similarly to \eqref{10a}, for $v_1$ that vanishes at $I \times \partial\Omega$ it holds that
\be \label{10b}
\|\partial_t(\cI_{\ell,k}^{P,1} v_1|_{P}-\invbreve{\cI}_{\ell,k}^{\cP,1} v_1)\|_{L_2(P)} \lesssim h_J^{-1} \max(h_J,h_K)^{j+1} |v_1|_{H^{j+1}(\omega_P)}.
\ee

\subsubsection{Definition and analysis of $\invbreve{\cI}_{\ell,k}^{\cP,2}$} \label{sec:modified}
Let $d \geq 2$. For some $s>2$, let ${\bf v}_2 \in L_1\big(I;L_s(\Omega)^d$  $\cap H(\divv;\Omega)\big)$.
Recalling that $\cI_{\ell,k}^{\cP,2} {\bf v}_2=(\cI_{\ell,k}^{P,2} {\bf v}_2|_P)_{P \in \cP}$, 
to ensure $L_2(I;H(\divv;\Omega))$-conformity of $\invbreve{\cI}_{\ell,k}^{\cP,2} {\bf v}_2$, 
for $F$ being a lateral slave facet of $P$ with corresponding master $P'$,
 in the DoFs of type
$\iint_{F} {\bf v}_2 (t,{\bf x})\cdot {\bf n}_{\bf x} \,q({\bf x}) p(t)\,d{\bf x}\,dt$ that define $\cI_{\ell,k}^{P,2} {\bf v}_2|_P$
we replace
 ${\bf v}_2|_F \cdot {\bf n}_{\bf x}$ by $(\cI_{\ell,k}^{P',2} {\bf v}_2|_{P'})|_F \cdot {\bf n}_{\bf x}$.
 Below, we derive local bounds for the quasi-interpolation error.
 
The DoFs for $\RT_k(K)$ given in \eqref{dof2} define a local quasi-interpolator as the sum over these DoFs of the product of the DoF and the associated dual basis function of $\RT_k(K)$.
If we restrict this sum to those DoFs that are associated to a facet $e$ of $K$, then the resulting operator, which we denote by $\cI_e$, satisfies
$$
\|\cI_e {\bf w}\|_{L_2(K)^d} \lesssim h_K^{d(\frac12-\frac{1}{s})} \|{\bf w}\|_{L_s(K)^d}+h_K \|\divv {\bf w}\|_{L_2(K)},
$$
as shown in \cite[Ch.~17, (17.20)]{70.97}.
Since $\cI_e {\bf w}$ depends on the normal trace of ${\bf w}$ on $e$ only, equally well it holds that
$$
\|\cI_e {\bf w}\|_{L_2(K)^d} \lesssim h_{K}^{d(\frac12-\frac{1}{s})} \|{\bf w}\|_{L_s(K'')^d}+h_K \|\divv {\bf w}\|_{L_2(K'')},
$$
where $K''$ is another uniformly shape-regular $d$-simplex that has $e$ as a face. By an application of an inverse inequality, it also shows that
\be \label{18}
\|\divv \cI_e {\bf w}\|_{L_2(K)} \lesssim h_{K}^{d(\frac12-\frac{1}{s})-1} \|{\bf w}\|_{L_s(K'')^d}+ \|\divv {\bf w}\|_{L_2(K'')}.
\ee

By extending this to the local quasi-interpolator on $P=J \times K$, with
$\cP_\ell(J) \otimes \RT_k(K)$ equipped with tensor DoFs built from \eqref{dof2a} and \eqref{dof2}, and lateral facet $F=J \times e$, we obtain 
for the restriction of this quasi-interpolator to the tensor DoF's associated to $F$, which we denote by $\cI_F$,
\be \label{2}
\|\cI_F {\bf w}\|_{L_2(P)^d} \lesssim h_{K}^{d(\frac12-\frac{1}{s})} \|{\bf w}\|_{L_2(J;L_s(K'')^d)}+h_K \|\divv_{\bf x} {\bf w}\|_{L_2(J \times K'')},
\ee
and
\be \label{3}
\|\divv_{\bf x} \cI_F {\bf w}\|_{L_2(P)} \lesssim h_{K}^{d(\frac12-\frac{1}{s})-1} \|{\bf w}\|_{L_2(J;L_s(K'')^d)}+ \|\divv_{\bf x} {\bf w}\|_{L_2(J \times K'')}.
\ee

\medskip

To proceed, it suffices to consider the situation that $P=J \times K$ has one lateral slave facet $J \times e$, where $P'=J' \times K'$ denotes the corresponding master.
 Then we can select a uniformly shape-regular $d$-simplex $K'' \subset K'$ that has facet $e$.
Recalling the definition of $\invbreve{\cI}_{\ell,k}^{\cP,2}$,  the application of \eqref{2} with ${\bf w}=(\identity - \cI_{\ell,k}^{P',2}){\bf v}_2$ yields
\be \label{8}
\begin{split}
&\|(\invbreve{\cI}_{\ell,k}^{\cP,2}  {\bf v}_2)|_{P} - \cI_{\ell,k}^{P,2}({\bf v}_2|_P)\|_{L_2(P)^d}  \\
&\quad\lesssim h_{K}^{d(\frac12-\frac{1}{s})} \|(\identity - \cI_{\ell,k}^{P',2}){\bf v}_2\|_{L_2(J';L_s(K')^d)}+h_K \|\divv_{\bf x} (\identity - \cI_{\ell,k}^{P',2}) {\bf v}_2\|_{L_2(P')}.
\end{split}
\ee
We estimate both terms in the upper bound separately.

Recalling that $\cI_{\ell,k}^{P',2}=Q_\ell^{J'} \otimes \rho_k^{K'}$, from $\divv \rho_k^{K'}=Q_k^{K'} \divv$ we have
$$
\divv_{\bf x} (\identity - \cI_{\ell,k}^{P',2})=\big((\identity -Q_\ell^{J'})\otimes \identity+Q_\ell^{J'} \otimes (\identity-Q_k^{K'})\big) \divv_{\bf x},
$$
which for $\ell'' \in \{0,\ldots,\ell+1\}$ and $k'' \in \{0,\ldots,k+1\}$ yields
\be \label{4}
\begin{split}
&\|\divv_{\bf x} (\identity - \cI_{\ell,k}^{P',2}) {\bf v}_2\|_{L_2(P')}\\
&\qquad \lesssim
h_{J'}^{\ell''} |\divv_{\bf x} {\bf v}_2|_{H^{\ell''}(J')\otimes L_2(K')}
+ h_K^{k''}  |\divv_{\bf x} {\bf v}_2|_{L_2(J')\otimes H^{k''}(K')}.
\end{split}
\ee

With the aid of \eqref{1}, valid for $s \in [1,\infty]$, one shows that for $k' \in \{0,\ldots,k\}$ and $\ell' \in \{0,\ldots,\ell+1\}$,
\be \label{5}
\begin{split}
&\|(\identity - \cI_{\ell,k}^{P',2}){\bf v}_2\|_{L_2(J';L_s(K')^d)}
\\ 
&\qquad\lesssim h_J^{\ell'} |{\bf v}_2|_{H^{\ell'}(J')\otimes L_s(K')^d}
+ h_K^{{k'}+1} |{\bf v}_2|_{L_2(J')\otimes W_s^{{k'}+1}(K')^d},
\end{split}
\ee
which generalizes this inequality for $s=2$ from Corollary~\ref{corol1}(iii).

By estimating
$$
\|(\identity-\invbreve{\cI}_{\ell,k}^{\cP,2}) {\bf v}_2\|_{L_2(P)^d} \leq \|(\identity - \cI_{\ell,k}^{P,2}){\bf v}_2\|_{L_2(P)^d}+\|(\invbreve{\cI}_{\ell,k}^{\cP,2}  {\bf v}_2)|_{P} - \cI_{\ell,k}^{P,2}({\bf v}_2|_P)\|_{L_2(P)^d},
$$
and using Corollary~\ref{corol1}(iii) to bound the first term, in combination with \eqref{8}, \eqref{4}, and \eqref{5} to bound the second one, we conclude that for $s > 2$,
$k''' \in \{0,\ldots,k\}$, and $\ell''' \in \{0,\ldots,\ell+1\}$,
\be \label{13}
\begin{split}
\|(\identity-&\invbreve{\cI}_{\ell,k}^{\cP,2}) {\bf v}_2\|_{L_2(P)^d} \lesssim 
h_J^{\ell'''} |{\bf v}_2|_{H^{\ell'''}(J)\otimes L_2(K)^d}
+ h_K^{{k'''}+1} |{\bf v}_2|_{L_2(J)\otimes H^{k'''+1}(K)^d}\\
&+ h_{K}^{d(\frac12-\frac{1}{s})}\Big[h_J^{\ell'} |{\bf v}_2|_{H^{\ell'}(J')\otimes L_s(K')^d)}
+ h_K^{{k'}+1} |{\bf v}_2|_{L_2(J')\otimes W_s^{k'+1}(K')^d}\Big]\\
& +h_K\Big[h_{J'}^{\ell''} |\divv_{\bf x} {\bf v}_2|_{H^{\ell''}(J')\otimes L_2(K')}
+ h_K^{k''} |\divv_{\bf x} {\bf v}_2|_{L_2(J')\otimes H^{k''}(K')} \Big].
\end{split}
\ee

The application of  \eqref{3} with ${\bf w}=(\identity - \cI_{\ell,k}^{P',2}){\bf v}_2$ gives
\be \label{19}
\begin{split}
\|\divv_{\bf x} (\invbreve{\cI}_{\ell,k}^{\cP,2} {\bf v}_2 - \cI_{\ell,k}^{P,2}{\bf v}_2|_P)\|_{L_2(P)} 
\lesssim &h_{K}^{d(\frac12-\frac{1}{s})-1} \|(\identity - \cI_{\ell,k}^{P',2}){\bf v}_2\|_{L_2(J';L_s(K')^d)}\\
&+ \|\divv_{\bf x} (\identity - \cI_{\ell,k}^{P',2}){\bf v}_2\|_{L_2(P')},
\end{split}
\ee
Using the bounds \eqref{5} and \eqref{4} for the norms at the right hand side, we conclude that
\be \label{14}
\begin{split}
\hspace*{-2em}\|\divv_{\bf x}(\invbreve{\cI}_{\ell,k}^{\cP,2} &{\bf v}_2 - \cI_{\ell,k}^{P,2}{\bf v}_2|_P)\|_{L_2(P)} 
\\
&\lesssim h_{K}^{d(\frac12-\frac{1}{s})-1} \Big[h_J^{\ell'} |{\bf v}_2|_{H^{\ell'}(J')\otimes L_{s}(K')^d}
+ h_K^{{k'}+1}  |{\bf v}_2(t,\cdot)|_{L_2(J')\otimes W_{s}^{{k'}+1}(K')^d}\Big]\\
&\quad + h_{J'}^{\ell''} |\divv_{\bf x} {\bf v}_2\|_{H^{\ell''}(J')\otimes L_2(K')}
+ h_K^{k''} |\divv_{\bf x} {\bf v}_2|_{L_2(J')\otimes H^{k''}(K')}.
\end{split}
\ee


\subsubsection{Wrap-up for the isotropic case} \label{sec:conclusions}
By taking $\ell+1=k$ in the isotropic case of $h_P :=h_J \eqsim h_K$, from \eqref{11} we have for $v_1$ that vanishes at $I \times \partial\Omega$,
\be \label{15}
\|\nabla_{\bf x}(v_1-\invbreve{\cI}_{k-1,k}^{\cP,1} v_1)\|_{L_2(P)^d} \lesssim h_{P}^k |v_1|_{H^{k+1}(\omega_P)}.
\ee
Recalling our assumption from Section~\ref{sec:modified} that $d \geq 2$,  by taking $s=s(d)=(\frac12-\frac{1}{d})^{-1} \in (2,\infty]$ and
with 
$\tilde{\Omega}_P=\tilde{\Omega}_P^\cP:=\{P' \in \cP\colon P \text{ has a lateral slave facet on } P'\}$
\be \label{16}
\begin{split}
\|(\identity-\invbreve{\cI}_{k-1,k}^{\cP,2}) {\bf v}_2&\|_{L_2(P)^d} \lesssim 
h_P^k\Big[ |{\bf v}_2|_{H^k(J)\otimes L_2(K)^d} +
  |{\bf v}_2|_{L_2(J)\otimes H^{k}(K)^d}+\\
\sum_{P'=J' \times K' \in\tilde{\Omega}_P}\Big\{  &|{\bf v}_2|_{H^{k-1}(J')\otimes L_s(K')^d}
+ |{\bf v}_2|_{L_2(J')\otimes W_s^{\max\{k-1,1\}}(K')^d}+\\
 &|\divv_{\bf x} {\bf v}_2|_{H^{k-1}(J')\otimes L_2(P')}
+  |\divv_{\bf x} {\bf v}_2|_{L_2(J')\otimes H^{k-1}(K')}\Big\} \Big].
\end{split}
\ee
Finally, the combination of \eqref{12}, Corollary~\ref{corol1}(i),  \eqref{10b}, and \eqref{14} shows that  for $v_1$ that vanishes at $I \times \partial\Omega$,
\be \label{17}
\begin{split}
 \|\divv ({\bf v} -\invbreve{\cI}_{k-1,k}^{\cP}  {\bf v})&\|_{L_2(P)}  \lesssim  h_P^k \Big[ |\divv {\bf v}|_{H^k(P)}+|v_1|_{H^{k+1}(\omega_P)}+\\
\sum_{P'=J'\times K' \in\tilde{\Omega}_P} \Big\{
& |{\bf v}_2|_{H^k(J')\otimes L_{s}(K')^d}
+ |{\bf v}_2|_{L_2(J')\otimes W_{s}^{k}(K')^d}+\\
& |\divv_{\bf x} {\bf v}_2|_{H^k(J')\otimes L_2(K')}
+ |\divv_{\bf x} {\bf v}_2|_{L_2(J')\otimes H^{k}(K')}
\Big\}\Big].
\end{split}
\ee
\emph{In conclusion:} Also for $P \in \cP$ for which the local patch $\omega_P^\cP$ is nonconforming, 
the sum of the upper bounds in \eqref{15}, \eqref{16}, and \eqref{17} gives an upper bound for $\|{\bf v} -\invbreve{\cI}_{k-1,k}^{\cP}  {\bf v}\|_{U,P}$ of order $h_P^k$, albeit under stronger smoothness assumptions than needed for \eqref{6} from Corollary~\ref{corol2} applicable when $\omega_P^\cP$ is conforming.

\subsubsection{The case $d=1$}
So far we have excluded the case $d=1$ from the definition and analysis of $\invbreve{\cI}^{\cP,2}_{\ell,k}$ in Section~\ref{sec:modified}, and so consequently from our conclusions presented in Section~\ref{sec:conclusions}.
Everything that was presented in Section~\ref{sec:modified}, however, is also valid for $d=1$, where the condition $s>2$ can even be read as $s\geq 2$.
Different from the $d \geq 2$ case, however, for $d=1$ the value of $s\in [2,\infty]$ cannot be chosen such that the exponent $d(\frac12-\frac{1}{s})-1$ in \eqref{18} is equal to $0$.
Via \eqref{3}, \eqref{13}, and \eqref{19}, it means that in \eqref{17} we do not obtain an upper bound of order $h_P^k$ unless we can take $\ell'=k+1$ (more precisely $\ell'=k+\frac12+\frac{1}{s}$)  in \eqref{5}. The latter however requires $\ell=k$ instead of $\ell+1=k$.

Taking therefore $\ell=k \geq 1$, i.e., $\cS_{k,k}(\cP)$ instead of $\cS_{k-1,k}(\cP)$ that was found most appropriate for the case of conforming partitions,  and $s=2$, similar to \eqref{15} and \eqref{16}, we obtain
\begin{align*}
\|\partial_x(v_1-\invbreve{\cI}_{k,k}^{\cP,1} v_1)\|_{L_2(P)} &\lesssim h_{P}^k |v_1|_{H^{k+1}(\omega_P)},\\
\|(\identity-\invbreve{\cI}_{k,k}^{\cP,2}) {\bf v_2}\|_{L_2(P)}  &\lesssim h_P^k \sum_{P' \in \tilde{\Omega}_P \cup\{P\}} |{\bf v}_2|_{H^k(P')},
\intertext{whereas instead of \eqref{17} we have}
\|\divv ({\bf v} -\invbreve{\cI}_{k,k}^{\cP}  {\bf v})\|_{L_2(P)}  &\lesssim  h_P^k \Big[ |\divv {\bf v}|_{H^k(P)}+|v_1|_{H^{k+1}(\omega_P)} + \sum_{P'\in\tilde{\Omega}_P} |{\bf v}_2|_{H^{k+1}(P')}\big].
\end{align*}

In our numerical experiments, however, we did not observe better results for $\cS_{k,k}(\cP)$ than for $\cS_{k-1,k}(\cP)$, and therefore we present only results for the latter option.


\section{Numerical results} \label{sec:numres}

\subsection{Adaptive algorithm}\label{sec:adaptive algorithm}

In the following numerical examples, we consider the least-squares formulation~\eqref{eq:first-order system} of the heat equation~\eqref{eq:heat} for $d=1$ and $d=2$. 
On prismatic partitions $\PP^\delta$ of the space-time cylinder $Q=I\times\Omega$ consisting of isotropic elements of the form $P=J \times K$ for closed intervals $J \subset I$ and closed $d$-simplices $K \subset \Omega$, we compute the corresponding least-squares approximation ${\bf u}^\delta:=\argmin_{{\bf v} \in \mathcal{S}_{0,1}(\PP^{\delta})}\|{\bf f}-G {\bf v}\|_L$ (i.e., $\ell+1=k$ as suggested in Remark~\ref{rem:ell and k}, and the lowest order case $k=1$). 
For convenience of the reader, we recall the definition of the discrete space
$$
\cS_{0,1}(\cP^\delta)=\{{\bf v}\in U\colon {\bf v}|_P \in \big(\cP_{1}(J) \otimes \cP_1(K)\big) \times \big(\cP_0(J) \otimes \RT_1(K)\big) \,(P \in \cP^\delta)\},
$$
where $\RT_1(K)$ should be read as $\cP_2(K)$ in the case that $d=1$.
Note that ${\bf u}^\delta$ can simply be computed as the solution of the finite-dimensional linear system
$$
\langle G{\bf u}^\delta, G{\bf v} \rangle_L = \langle {\bf f}, G {\bf v} \rangle_L \quad ({\bf v}\in \mathcal{S}_{0,1}(\PP^\delta)),
$$
where $\langle\cdot,\cdot\rangle_L$ denotes the scalar product on the $L_2$-type space $L=L_2(Q)\times L_2(Q)^d\times L_2(\Omega)$.
 
To measure the error, we use the reliable and efficient {\sl a posteriori} error estimator 
$$
 \eta({\bf f},{\bf v}):=  
 \| {\bf f} - G {\bf v} \|_{L} \eqsim \| {\bf u} - {\bf v}\|_U \quad ({\bf v}\in U),
$$
with corresponding local error indicators
$$
 \eta(P; {\bf f}, {\bf v}):=  \| {\bf f} - G {\bf v} \|_{L(P)} 
 \quad(P\in\PP^\delta),
$$
where
\begin{align*}
 L(P) := L_2(P)\times L_2(P)^d \times L_2(\partial P \cap (\{0\}\times \Omega)).
\end{align*}

While the a priori estimates of Section~\ref{sec:global space} only guarantee the convergence rate $\mathcal{O}({\rm dofs}^{-\frac1{d+1}})$ for a sequence of quasi-uniform meshes and sufficiently smooth solution ${\bf u}$, 
we will also investigate adaptive refinement of elements $\mathcal{M}^\delta\subseteq\PP^\delta$ with ``large'' error indicator for singular solutions ${\bf u}$. 
We use D\"orfler marking $ \theta\eta({\bf f},{\bf u}^\delta)^2 \le  \sum_{P\in\mathcal{M}^\delta}\eta(P;{\bf f}, {\bf u}^\delta)^2$ with some $0<\theta\le1$ to determine the marked elements. 

Starting from a conforming initial partition of $\overline{Q}$ into (closed) prisms, we consider the following refinement procedure: In the spatial direction a $d$-simplex to be refined is decomposed into $2^d$-simplices by some rule such that a uniform refinement of a conforming partition of $\overline{\Omega}$ results in a conforming decomposition, and a repeated application of this rule produces uniformly shape-regular simplices, and we combine this rule with bisection in the temporal direction.
Now the level of each prism that is produced is defined as the number of these product decompositions needed to create the prism from a prism in the initial partition, whose level is set to zero.
All prisms that can be produced in this way are uniformly shape-regular, and so in particular (uniformly) isotropic.

In our experiments for spatial domains of dimension $d=1$ or $d=2$, as the rule to refine a $d$-simplex we applied bisection or `red'-refinement, respectively.

In the case of local refinements, apart from the prisms marked for refinement, a minimal number of other prisms are refined such that the difference in levels of any two prisms in the new partition that have non-empty intersection is less or equal to one.
This condition guarantees that all partitions that can be created are $1$-irregular.


Overall, we thus consider the following adaptive algorithm.

\begin{algorithm}
\label{alg:adaptive}
\textbf{Input:} 
Right-hand side ${\bf f}\in L$, initial conforming mesh $\PP^0=\PP^{\delta_0}$, D\"orfler parameter $0<\theta\le1$.
\\
\textbf{Loop:} For each $j=0,1,2,\dots$, iterate the following steps {\rm(i)--(iv)}:
\begin{enumerate}[\rm(i)]
\item  Compute least-squares approximation ${\bf u}^j={\bf u}^{\delta_j} = \argmin_{{\bf v} \in \mathcal{S}_{0,1}(\PP^{j})}\|{\bf f}-G {\bf v}\|_L$ of ${\bf u}=(u,-\nabla_{\bf x} u-{\bf f}_2)$. %
\item Compute  error indicators $\eta({P;{\bf f}, {\bf u}^j})$ for all elements ${P}\in\PP^j=\PP^{\delta_j}$. %
\item\label{item:marking}  Determine a minimal set of marked elements $\mathcal{M}^j\subseteq\PP^j$ satisfying D\"orfler marking 
\begin{align*}
 \theta\,\eta({\bf f},{\bf u}^j)^2 \le  \sum_{P\in\mathcal{M}^j}\eta(P;{\bf f}, {\bf u}^j)^2.
\end{align*}
\item Generate refined prismatic mesh $\PP^{j+1}$ by refining at least all marked elements $\mathcal{M}^j$ as explained above. 
\end{enumerate}
\textbf{Output:} Refined meshes $\PP^j$, corresponding exact discrete solutions ${\bf u}^j$, and 
error estimators $\eta({\bf f}, {\bf u}^j)$ for all $j \in \N_0$.
\end{algorithm}

For comparison, we also consider the adaptive algorithm of \cite{75.257}, i.e., Algorithm~\ref{alg:adaptive} for conforming simplicial meshes $\TT^\delta$ instead of prismatic meshes $\PP^\delta$, refined by newest vertex bisection~\cite{249.87}, where, as in \cite{75.257}, we refine marked elements into 
$2^{d+1}$ elements by repeated bisection, with associated finite element space 
$$
\{{\bf v}\in C^0(Q)\times C^0(Q)^d\colon {\bf v}|_T \in \cP_1(T) \times  \cP_1(T)^d \,(T \in \TT^\delta), v_1|_{I\times\partial\Omega} = 0\}\subset U.
$$

Although, we present numerical comparisons only for some of the test examples of \cite{75.257} (plus some new ones),
we emphasize that we computed numerical results for all these examples and observed that in each case our prismatic finite element space $\cS_{0,1}(\cP^\delta)$ gave either better or the same convergence rates of the estimator, for both uniform refinement, where in each step all elements are marked for refinement, and adaptive refinement.

\begin{remark}
In \cite{75.28}, we proved plain convergence 
$$
\lim_{j\to\infty}\eta({\bf f},{\bf u}^j) = 0= \lim_{j\to\infty}\|{{\bf u}-\bf u}^j\|_U
$$ 
of the adaptive algorithm from \cite{75.257}, even for arbitrary polynomial degree of the employed discrete functions. 
Exploiting the local approximation properties derived in Section~\ref{sec:general prisms}, the proof easily extends to Algorithm~\ref{alg:adaptive} (with prismatic meshes), even for arbitrary polynomial degrees $\ell\in\N_0$, $k\in\N$ in the discrete space $\mathcal{S}_{\ell,k}(\PP^j)$ (instead of $\mathcal{S}_{0,1}(\PP^j)$). 
Convergence is even satisfied if the marking step~\eqref{item:marking} is replaced by  determining a (not necessarily minimal) set of  marked elements $\mathcal{M}^j\subseteq\PP^j$ with the property
$$
 \max_{P\in\PP^j\setminus\mathcal{M}^j} \eta(P;{\bf f}, {\bf u}^j) \le M(\max_{P\in\mathcal{M}^j} \eta(P;{\bf f}, {\bf u}^j))
$$
for some fixed marking function $M:[0,\infty)\to[0,\infty)$ that is continuous at $0$ with $M(0)=0$.
\end{remark}

\subsection{Experiments in 1+1D}
For the following experiments, we fix the domain $\Omega:=(0,1)$ and the end time point $T:=1$. 
As initial prismatic mesh of the resulting space-time cylinder $Q$, we take $\PP^0=\{\overline Q\}$. 
For the initial triangular mesh $\TT^0$, we split $\overline Q$ into four  triangles. 
We consider uniform refinement, where in each step all elements are marked for refinement, and adaptive refinement with  D\"orfler parameter $\theta=0.5$.  

\subsubsection{Non-matching initial datum}\label{sec:2d_f12_u01}
We consider \cite[Example 4]{75.257}, i.e., 
$$
{\bf f}=(f_1,{\bf f}_2,u_0)=(2,0,1).
$$ 
Note that the initial datum does not match with the homogenous Dirichlet boundary conditions so that the solution $u$ of the heat equation and the associated ${\bf u}=(u,-\nabla_x u)$ are singular for $(t,x)\in\{(0,0),(0,1)\}$. 
Figure~\ref{fig:2d_f12_u01} displays the resulting error estimators $\eta$ for simplicial and prismatic meshes, and uniform and adaptive refinement vs.~the degrees of freedom.  
The rates $0.08$ and $0.17$ for simplicial meshes have been already reported in \cite{75.257}. 
For prismatic meshes, we observe significantly improved rates $0.13$ and $0.43$. 
As mentioned before, the optimal rate for smooth solutions would be $1/2$.

\begin{center}
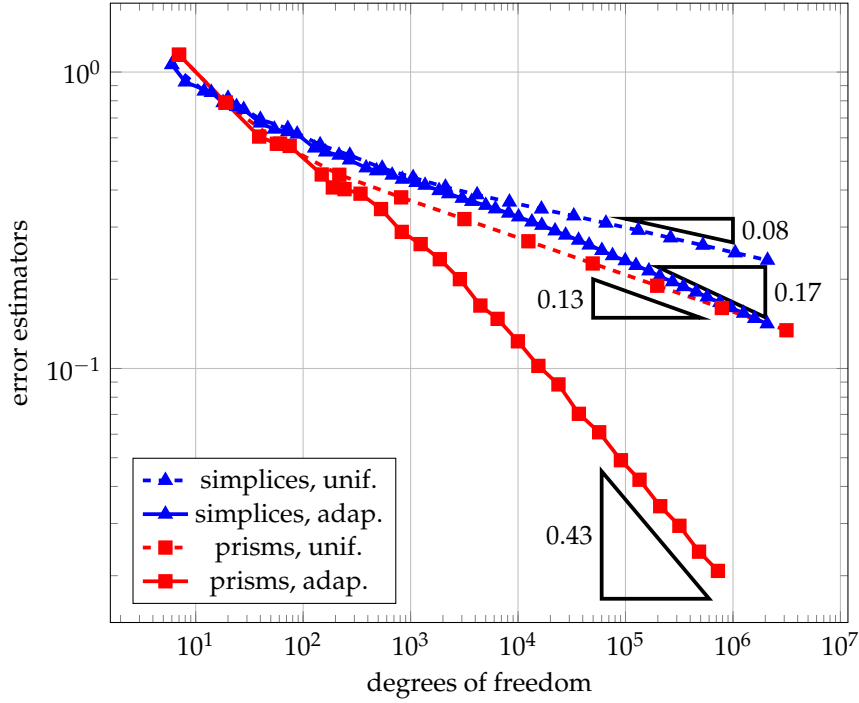
\begin{figure}
\begin{tikzpicture}
\begin{loglogaxis}[
width = 0.9\textwidth,
clip = true,
xlabel= {degrees of freedom},
ylabel= {error estimators},
grid = major,
legend pos = south west,
legend entries = {simplices, unif. \\simplices, adap. \\ prisms, unif. \\prisms, adap. \\} 
]
\addplot[line width = 0.5mm,color=blue,mark=triangle*,dashed,mark options=solid] table[x=dofs,y=estimators, col sep=comma] {data/simp2d/f2_u01_theta100.csv};
\addplot[line width = 0.5mm,color=blue,mark=triangle*] table[x=dofs,y=estimators, col sep=comma] {data/simp2d/f2_u01_theta50.csv};
\addplot[line width = 0.5mm,color=red,mark=square*,dashed,mark options=solid] table[x=dofs,y=estimators, col sep=comma] {data/prism2d/f2_u01_pxs2_theta100.csv};
\addplot[line width = 0.5mm,color=red,mark=square*] table[x=dofs,y=estimators, col sep=comma] {data/prism2d/f2_u01_pxs2_theta50.csv};

\convtri{100000}{0.32}{0.08}{0.08}
\convtri{200000}{0.22}{0.17}{0.17}
\convtriinv{50000}{0.2}{0.13}{0.13}
\convtriinv{60000}{0.045}{0.43}{0.43}

\end{loglogaxis}
\end{tikzpicture}
\caption{\label{fig:2d_f12_u01} Convergence plot of Section~\ref{sec:2d_f12_u01}  for spatial domain $\Omega=(0,1)$, right-hand side $(f_1,{\bf f}_2,u_0) = (2,0,1)$, simplicial and prismatic meshes, and uniform and adaptive refinement.}
\end{figure}
\end{center}

\subsubsection{Initial datum with a ``singularity'' in the interior}\label{sec:2d_f10_u0hat}
We consider \cite[Example~2]{75.257}, i.e., 
$$
(f_1,{\bf f}_2,u_0) = \big(1,0,x\mapsto 1-2\big|x-\tfrac12\big|\big).
$$
Clearly, the (at $x=\tfrac12$)  irregular initial datum yields a reduced regularity of the unknown solution  $u$ of the heat equation and the associated ${\bf u}=(u,-\nabla_x u)$. 
Figure~\ref{fig:2d_f10_u0hat} displays the resulting error estimators $\eta$ for simplicial and prismatic meshes, and uniform and adaptive refinement over the degrees of freedom.  
The rates $0.25$ and $0.42$ for simplicial meshes have been already observed in \cite{75.257}. 
Only adaptive refinement for prismatic meshes leads to the optimal rate $1/2$, while uniform refinement yields the rate $0.38$.

\begin{center}
\begin{figure}
\begin{tikzpicture}
\begin{loglogaxis}[
width = 0.9\textwidth,
clip = true,
xlabel= {degrees of freedom},
ylabel= {error estimators},
grid = major,
legend pos = south west,
legend entries = {simplices, unif. \\simplices, adap. \\ prisms, unif. \\prisms, adap. \\} 
]
\addplot[line width = 0.5mm,color=blue,mark=triangle*,dashed,mark options=solid] table[x=dofs,y=estimators, col sep=comma] {data/simp2d/f1_u0hat_theta100.csv};
\addplot[line width = 0.5mm,color=blue,mark=triangle*] table[x=dofs,y=estimators, col sep=comma] {data/simp2d/f1_u0hat_theta50.csv};
\addplot[line width = 0.5mm,color=red,mark=square*,dashed,mark options=solid] table[x=dofs,y=estimators, col sep=comma] {data/prism2d/f1_u0hat_pxs2_theta100.csv};
\addplot[line width = 0.5mm,color=red,mark=square*] table[x=dofs,y=estimators, col sep=comma] {data/prism2d/f1_u0hat_pxs2_theta50.csv};

\convtri{200000}{0.04}{0.25}{0.25}
\convtriinv{700000}{0.006}{0.38}{0.38}
\convtri{300000}{0.016}{0.42}{0.42}
\convtriinv{200000}{0.0027}{0.5}{0.5}

\end{loglogaxis}
\end{tikzpicture}
\caption{\label{fig:2d_f10_u0hat} Convergence plot of Section~\ref{sec:2d_f10_u0hat}  for spatial domain $\Omega=(0,1)$, right-hand side $(f_1,{\bf f}_2,u_0) = \big(1,0,x\mapsto 1-2\big|x-\tfrac12\big|\big)$, simplicial and prismatic meshes, and uniform and adaptive refinement.}
\end{figure}
\end{center}

\subsubsection{Initial datum with a ``singularity'' at the boundary}\label{sec:2d_f10_u0x05}
We consider
$$
(f_1,{\bf f}_2,u_0) = (0,0,x\mapsto x^{1/2}(1-x))
$$
Clearly, the (at $x=0$) irregular initial datum yields a reduced regularity of the unknown solution $u$ of the heat equation and the associated ${\bf u}=(u,-\nabla_x u)$. 
Figure~\ref{fig:2d_f10_u0x05} displays the resulting error estimators $\eta$ for simplicial and prismatic meshes, and uniform and adaptive refinement over the degrees of freedom.  
As in the previous example, only adaptive refinement for prismatic meshes leads to the optimal rate $1/2$, while uniform refinement yields the rate $0.26$.
For simplicial meshes, we observe the rates $0.19$ and $0.33$. 

\begin{center}
\begin{figure}
\begin{tikzpicture}
\begin{loglogaxis}[
width = 0.9\textwidth,
clip = true,
xlabel= {degrees of freedom},
ylabel= {error estimators},
grid = major,
legend pos = south west,
legend entries = {simplices, unif. \\simplices, adap. \\ prisms, unif. \\prisms, adap. \\} 
]
\addplot[line width = 0.5mm,color=blue,mark=triangle*,dashed,mark options=solid] table[x=dofs,y=estimators, col sep=comma] {data/simp2d/f0_u0x05_theta100.csv};
\addplot[line width = 0.5mm,color=blue,mark=triangle*] table[x=dofs,y=estimators, col sep=comma] {data/simp2d/f0_u0x05_theta50.csv};
\addplot[line width = 0.5mm,color=red,mark=square*,dashed,mark options=solid] table[x=dofs,y=estimators, col sep=comma] {data/prism2d/f0_u0x05_pxs2_theta100.csv};
\addplot[line width = 0.5mm,color=red,mark=square*] table[x=dofs,y=estimators, col sep=comma] {data/prism2d/f0_u0x05_pxs2_theta50.csv};

\convtri{100000}{0.045}{0.19}{0.19}
\convtriinv{100000}{0.011}{0.33}{0.33}
\convtri{300000}{0.014}{0.26}{0.26}
\convtriinv{100000}{0.002}{0.5}{0.5}

\end{loglogaxis}
\end{tikzpicture}
\caption{\label{fig:2d_f10_u0x05} Convergence plot of Section~\ref{sec:2d_f10_u0x05}  for spatial domain $\Omega=(0,1)$, right-hand side $(f_1,{\bf f}_2,u_0) = (0,0,x\mapsto x^{1/2}(1-x))$, simplicial and prismatic meshes, and uniform and adaptive refinement.}
\end{figure}
\end{center}

\subsection{Experiments in 2+1D}
For the following experiments, we fix the domain $\Omega:=(0,1)^2$ and the end time point $T:=1$. 
The initial prismatic mesh of the resulting space-time cylinder $Q$ is obtained by splitting $\overline Q$ into two prisms with triangular base. 
For the initial tetrahedral mesh $\TT^0$, we split $\overline Q$ into twelve tetrahedra. 
We consider uniform refinement, where in each step all elements are marked for refinement, and adaptive refinement with  D\"orfler parameter $\theta=0.5$.

\subsubsection{Non-matching initial datum}\label{sec:3d_f10_u01}
We consider \cite[Example 7]{75.257}, i.e., 
$$
{\bf f}=(f_1,{\bf f}_2,u_0)=(0,0,1).
$$ 
Note that the initial datum does not match with the homogenous Dirichlet boundary conditions so that the solution $u$ of the heat equation and the associated ${\bf u}=(u,-\nabla_x u)$ are singular for $(t,x)\in \{0\}\times\partial\Omega$. 
Figure~\ref{fig:3d_f10_u01} displays the resulting error estimators $\eta$ for simplicial and prismatic meshes, and uniform and adaptive refinement vs.~the degrees of freedom.  
The rates $0.07$ and $0.08$ for simplicial meshes have been already reported in \cite{75.257}. 
For prismatic meshes, we observe somewhat
improved rates $0.09$ and $0.14$. 
As mentioned before, the optimal rate for smooth solutions would be $1/3$.
The reason for the low rates even for adaptively refined prismatic meshes is likely due to the edge singularities of the solution requiring anisotropic refinement in both space and time.

\begin{center}
\begin{figure}
\begin{tikzpicture}
\begin{loglogaxis}[
width = 0.9\textwidth,
clip = true,
xlabel= {degrees of freedom},
ylabel= {error estimators},
grid = major,
legend pos = south west,
legend entries = {simplices, unif. \\simplices, adap. \\ prisms, unif. \\prisms, adap. \\} 
]
\addplot[line width = 0.5mm,color=blue,mark=triangle*,dashed,mark options=solid] table[x=dofs,y=estimators, col sep=comma] {data/simp3d/f0_u01_theta100.csv};
\addplot[line width = 0.5mm,color=blue,mark=triangle*] table[x=dofs,y=estimators, col sep=comma] {data/simp3d/f0_u01_theta50.csv};
\addplot[line width = 0.5mm,color=red,mark=square*,dashed,mark options=solid] table[x=dofs,y=estimators, col sep=comma] {data/prism3d/f0_u01_pxs2_theta100.csv};
\addplot[line width = 0.5mm,color=red,mark=square*] table[x=dofs,y=estimators, col sep=comma] {data/prism3d/f0_u01_pxs2_theta50.csv};

\convtriinv{100000}{0.37}{0.14}{0.14}
\convtriinv{100000}{0.37}{0.14}{0.14}
\convtri{100000}{0.65}{0.07}{0.07}
\convtri{300000}{0.47}{0.09}{0.09}

\end{loglogaxis}
\end{tikzpicture}
\caption{\label{fig:3d_f10_u01} Convergence plot of Section~\ref{sec:3d_f10_u01}  for spatial domain $\Omega=(0,1)^2$, right-hand side $(f_1,{\bf f}_2,u_0) = (0,0,1)$, simplicial and prismatic meshes, and uniform and adaptive refinement.}
\end{figure}
\end{center}

\subsubsection{Initial datum with a ``singularity'' at an interior point}\label{sec:3d_f0_u0r1}
We consider
$$
(f_1,{\bf f}_2,u_0) = \big(0,0,{\bf x}\mapsto \sqrt{(x_1-\tfrac12)^2+(x_2-\tfrac12)^2} \sin(\pi x_1) \sin(\pi x_2)\big). 
$$
Clearly, the (at ${\bf x}=(\tfrac12,\tfrac12)$) irregular initial datum yields a reduced regularity of the unknown solution  $u$ of the heat equation and the associated ${\bf u}=(u,-\nabla_x u)$. 
Figure~\ref{fig:3d_f0_u0r1} displays the resulting error estimators $\eta$ for simplicial and prismatic meshes, and uniform and adaptive refinement over the degrees of freedom.  
Only adaptive refinement for prismatic meshes leads to the optimal rate $1/3$, while uniform refinement yields the rate $0.27$.
For simplicial meshes, we observe the rates $0.13$ and~$0.18$. 

\begin{center}
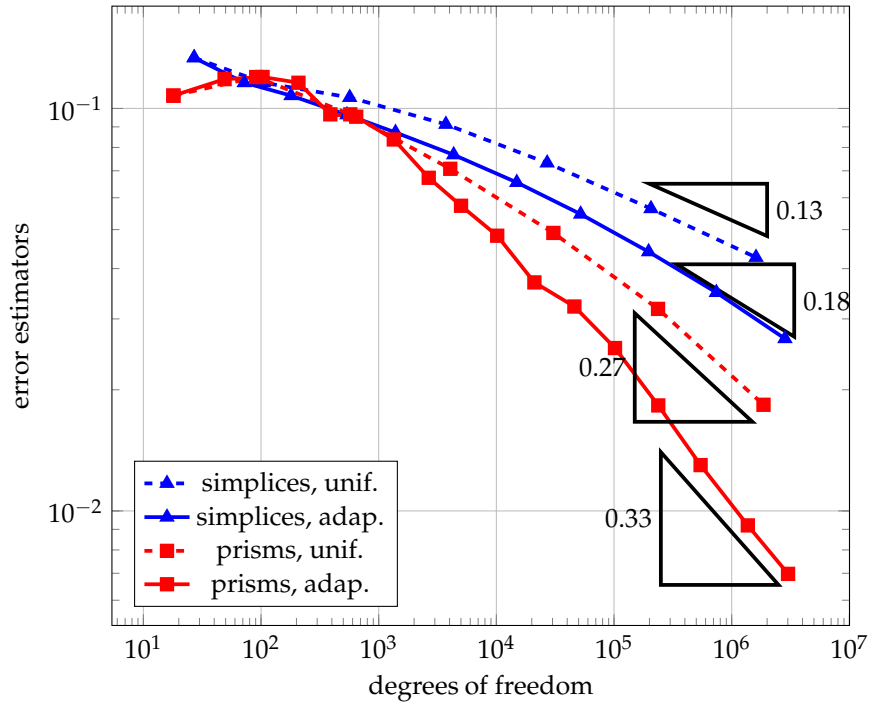
\begin{figure}
\begin{tikzpicture}
\begin{loglogaxis}[
width = 0.9\textwidth,
clip = true,
xlabel= {degrees of freedom},
ylabel= {error estimators},
grid = major,
legend pos = south west,
legend entries = {simplices, unif. \\simplices, adap. \\ prisms, unif. \\prisms, adap. \\} 
]
\addplot[line width = 0.5mm,color=blue,mark=triangle*,dashed,mark options=solid] table[x=dofs,y=estimators, col sep=comma] {data/simp3d/f0_u0r1_theta100.csv};
\addplot[line width = 0.5mm,color=blue,mark=triangle*] table[x=dofs,y=estimators, col sep=comma] {data/simp3d/f0_u0r1_theta50.csv};
\addplot[line width = 0.5mm,color=red,mark=square*,dashed,mark options=solid] table[x=dofs,y=estimators, col sep=comma] {data/prism3d/f0_u0r1_pxs2_theta100.csv};
\addplot[line width = 0.5mm,color=red,mark=square*] table[x=dofs,y=estimators, col sep=comma] {data/prism3d/f0_u0r1_pxs2_theta50.csv};

\convtri{200000}{0.065}{0.13}{0.13}
\convtri{340000}{0.041}{0.18}{0.18}
\convtriinv{150000}{0.031}{0.27}{0.27}
\convtriinv{250000}{0.014}{0.33}{0.33}

\end{loglogaxis}
\end{tikzpicture}
\caption{\label{fig:3d_f0_u0r1} Convergence plot of Section~\ref{sec:3d_f0_u0r1}  for spatial domain $\Omega=(0,1)^2$, right-hand side $(f_1,{\bf f}_2,u_0) = \big(0,0, {\bf x}\mapsto \sqrt{(x_1-\tfrac12)^2+(x_2-\tfrac12)^2} \sin(\pi x_1) \sin(\pi x_2)\big) 
$, simplicial and prismatic meshes, and uniform and adaptive refinement.}
\end{figure}
\end{center}

\subsubsection{Initial datum with a ``singularity'' at a boundary edge}\label{sec:3d_f10_u0x075}
We consider
$$
(f_1,{\bf f}_2,u_0) = (0,0,{\bf x}\mapsto x_1^{3/4}(1-x_1)x_2(1-x_2))
$$
Clearly, the (at $x_1=0$) irregular initial datum yields a reduced regularity of the unknown solution  $u$ of the heat equation and the associated ${\bf u}=(u,-\nabla_x u)$. 
Figure~\ref{fig:3d_f10_u0x075} displays the resulting error estimators $\eta$ for simplicial and prismatic meshes, and uniform and adaptive refinement over the degrees of freedom.  
As in the previous example, only adaptive refinement for prismatic meshes leads to the optimal rate $1/3$, while uniform refinement yields the rate $0.3$.
For simplicial meshes, we observe the rates $0.23$ and~$0.27$. 

\begin{center}
\begin{figure}
\begin{tikzpicture}
\begin{loglogaxis}[
width = 0.9\textwidth,
clip = true,
xlabel= {degrees of freedom},
ylabel= {error estimators},
grid = major,
legend pos = south west,
legend entries = {simplices, unif. \\simplices, adap. \\ prisms, unif. \\prisms, adap. \\} 
]
\addplot[line width = 0.5mm,color=blue,mark=triangle*,dashed,mark options=solid] table[x=dofs,y=estimators, col sep=comma] {data/simp3d/f0_u0x075_theta100.csv};
\addplot[line width = 0.5mm,color=blue,mark=triangle*] table[x=dofs,y=estimators, col sep=comma] {data/simp3d/f0_u0x075_theta50.csv};
\addplot[line width = 0.5mm,color=red,mark=square*,dashed,mark options=solid] table[x=dofs,y=estimators, col sep=comma] {data/prism3d/f0_u0x075_pxs2_theta100.csv};
\addplot[line width = 0.5mm,color=red,mark=square*] table[x=dofs,y=estimators, col sep=comma] {data/prism3d/f0_u0x075_pxs2_theta50.csv};

\convtri{200000}{0.016}{0.23}{0.23}
\convtri{460000}{0.008}{0.27}{0.27}
\convtriinv{45000}{0.008}{0.3}{0.3}
\convtriinv{100000}{0.0037}{0.33}{0.33}

\end{loglogaxis}
\end{tikzpicture}
\caption{\label{fig:3d_f10_u0x075} Convergence plot of Section~\ref{sec:3d_f10_u0x075}  for spatial domain $\Omega=(0,1)^2$, right-hand side $(f_1,{\bf f}_2,u_0) = (0,0,{\bf x}\mapsto x_1^{3/4}(1-x_1)x_2(1-x_2))$, simplicial and prismatic meshes, and uniform and adaptive refinement.}
\end{figure}
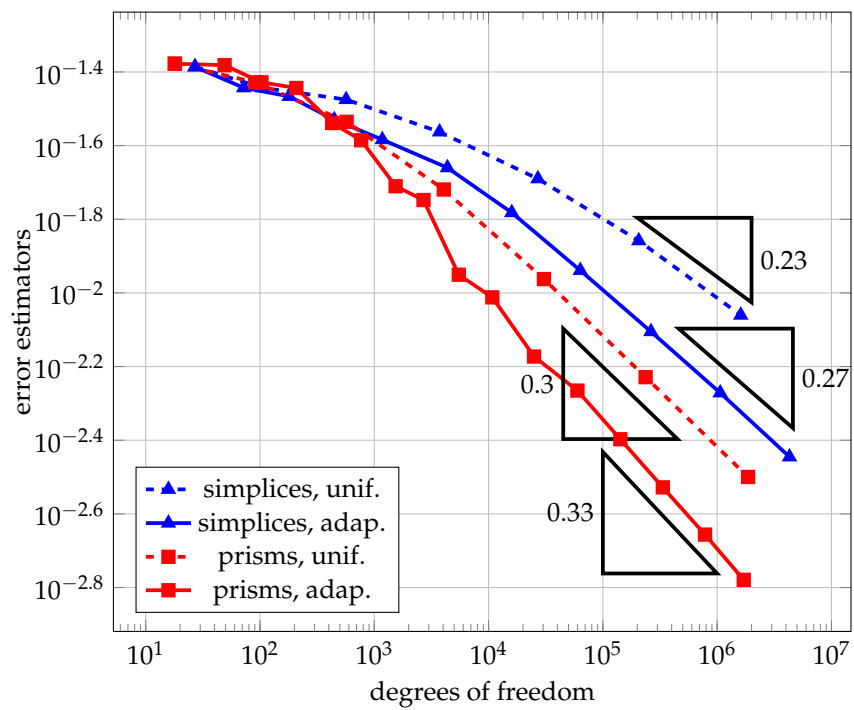
\end{center}

\bibliographystyle{alpha}
\bibliography{../ref}
\end{document}